\documentclass[letterpaper,12pt]{amsart}
\usepackage{amsmath}
\usepackage{amssymb}
\usepackage{fix-cm}
\usepackage[left=1in,top=1in,right=1in,bottom=1in]{geometry}
\usepackage{hyperref}
\usepackage{amsthm}
\usepackage{graphicx}

\newcommand{\comments}[1]{}

\newtheorem{theorem}{Theorem}[section]

\newtheorem{lemma}[theorem]{Lemma}

\DeclareGraphicsExtensions{.eps}

\newtheorem{utv*}{Proposition}
\newtheorem{hyp*}{Conjecture}

\newtheorem{zamech}{Remark}
\newtheorem{zamech*}{Remark}
\newtheorem*{th*}{Theorem}

\numberwithin{equation}{section}

\newcommand{\norm}[2]{\left\| #1 \right\|_{#2}}
\newcommand{\mdl}[1]{\left| #1 \right|}
\newcommand{\ave}[2]{m_{#1} #2}
\newcommand{\av}[2]{m_{#2} #1}

\newcommand{\R}{\mathbb{R}}

\def\sli{\sum\limits}
\def\ili{\int\limits}

\def\gp{\gamma_+}

\def\ep{\varepsilon}
\def\vf{\varphi}

\def\Om{\Omega}

\title{Sharp estimates involving $A_\infty$ and $LlogL$ constants, and their applications to PDE}
\author{O. Beznosova}
\address{Department of Mathematics, Baylor University, One Bear Place \#97328, Waco, TX 76798-7328, USA.}
\author{A. Reznikov}
\address{Department of Mathematics, Michigan State University, East
Lansing, MI 48824, USA}
\address{St.-Petersburg Department of the Steklov Mathematical Institute, Fontanka, 27, 191023, Saint Petersburg, Russia.}

\begin{document}

\maketitle

\begin{abstract}
{It is a well known fact that the union of the Reverse H\"{o}lder classes, $\bigcup_{p>1} RH_p$ coincides with the union of the Muckenhoupt classes $\bigcup_{p>1} A_p = A_\infty$, but the $A_\infty$ constant of the weight $w$, which is a limit of its $A_p$ constants, is not a natural characterization for the weight in Reverse H\"{o}lder classes. We introduce the
$RH_1$ condition as a limiting case of the $RH_p$ inequalities as $p$ tends to $1$, show sharp bound on $RH_1$ constant of the weight $w$ in terms of its $A_\infty$ constant. We also prove the sharp version of the Gehring theorem for the case $p=1$, completing the answer to the famous question of Bojarski in dimension one, see \cite{Boiarsky:1985}.

We illustrate our results by two straight-forward applications: to the Dirichlet problem for elliptic PDE's.

To prove our main theorem we are going to use the Bellman function technique. We do it in the spirit of the paper \cite{Vasyuninrus:04}. However, to simplify our calculations, we will use the Monge-Ampere equation and some intuition from papers \cite{Vasyuninrus:04}, \cite{SlavinVasyunin}. In the same spirit we find the ``extremal'' function $w$.

Despite the fact that our methods are not new, we believe that our results are useful, thus we prove them in full details.

}
\end{abstract}

\section{Definitions and Main Results.}
\label{s: intro}

We say that $w$ is a {\it weight} if it is a locally integrable function on the real line,
positive almost everywhere (with respect to the Lebesgue measure).
Let $\ave{J}{w}$ be the average of a weight $w$ over a given interval $J \subset \R$:
$$
\ave{J}{w} := \frac{1}{\mdl{J}} \int_J w dx.
$$
A weight $w$ belongs to the \textit{Muckenhoupt class} $A_p$ whenever its Muckenhoupt constant $[w]_{A_p}$ is finite:
\begin{equation}\label{defAp}
[w]_{A_p} := \sup_{J\subset\R} \; \ave{J}{w} \left(\ave{J}{\left(w^{-\frac{1}{p-1}}\right)}\right)^{p-1} < \infty.
\end{equation}

Note that by H\"{o}lders inequality, $[{w}]_{A_p} \geqslant 1$ holds for all $1<p<\infty$, as
well as the following inclusion:
$$
if \;\;\; 1<p\leqslant q < \infty \;\;\; then \;\;\; A_p\; \subseteq
A_q, \; \; \; \; [{w}]_{A_q} \; \leqslant \; [{w}]_{A_p}.
$$
So, for $1<p<\infty$ Muckenhoupt classes $A_p$ form an increasing chain. There are two natural limits of it  - as $p$ approaches $1$ and as $p$ goes to $\infty$. We will be interested in the limiting case as $p \rightarrow \infty$, $A_\infty = \bigcup_{p>1} A_p$. There are several equivalent definitions of it, we will state one that we are going to use (the natural limit of $A_p$ conditions, that also defines the $A_\infty$ constant of the weight $w$), for other equivalent definitions see \cite{GarciaCuervaRubioDeFrancia:85}, \cite{Grafakos:03} or \cite{Stein:93}.
\begin{equation}\label{defAinfty}
w \in A_\infty  \;\;\;\;\;\;\Longleftrightarrow\;\;\;\;\;\;[w]_{A_\infty} :=\;\;\; \sup_{J\subset \R} \;\;\ave{J}{w} \;e^{-\ave{J}{(\log w)}} \;\; < \infty.
\end{equation}

A weight $w$ belongs to the {\it Reverse H\"{o}lder class} $RH_p$ ($1<p<\infty$) if
\begin{equation}\label{defRHp}
[w]_{RH_p}\;:=\;\sup_{J\subset \R} \;\frac{\left( m_Jw^p \right)^{1/p}}{ m_Jw} < \infty.
\end{equation}

Note that by H\"{o}lders inequality the Reverse H\"{o}lder classes satisfy:
$$
if \;\;\; 1<p\leqslant q<\infty, \; \; \; then \; \; \; RH_q \;
\subseteq \; RH_p \;\;\; and \;\;\; 1 \; \leqslant \; [w]_{RH_p}
\; \leqslant \; [w]_{RH_q},
$$
which is similar to the inclusion chain of the $A_p$ classes, except inclusion runs in the opposite direction. And similarly we can consider two limiting cases $RH_\infty$ (the smallest) and $RH_1$ (the largest). Same as in the case of Muckenhoupt classes we are more interested in the largest one, let us call it $RH_1 := \bigcup_{p>1} RH_p$.

For the $A_\infty$ and $RH_1$ in 1974 Coifman and Fefferman showed that $A_\infty = \bigcup_{p>1} RH_p = RH_1$. Now it is a well known fact (see \cite{GarciaCuervaRubioDeFrancia:85}, \cite{Grafakos:03}, \cite{Stein:93}) that if $w \in A_p$ then $w \in RH_q$ for some $1< q< \infty$ and vice versa. In  \cite{Grafakos:03}  dependencies of $p$ and $q$ and of $A_p$ and $RH_q$ constants in any dimension are traced roughly.  The $A_1$ and $RH_\infty$ classes are not overlooked either, a lot of information about them can be found in \cite{Cruz-UribeNeugenbauer:89}.  Exact dependencies are much harder to trace, but for $1\leqslant p \leqslant \infty$ and $1 < q \leqslant \infty$ in one dimensional case precise dependencies between $A_p$ and $RH_q$ are found in \cite{Vasyuninrus:04}.

The question is : Is anything missing in the precise relationships between $A_p$ and $RH_q$ constants?

The answer is ``Yes'' and let us now describe the missing little piece of this puzzle.

Union of Reverse H\"{o}lder classes is $A_\infty$, but the $A_\infty$ constant (the natural limit of $A_p$ constants) has nothing to do with the Reverse H\"{o}lder constants.
The natural limit as $p \rightarrow 1^+$ of the Reverse H\"{o}lder inequalities is the following condition, which we will take as a definition of the class $RH_1$:
\begin{equation}\label{defRH1}
w \in RH_1 \;\;\; \Longleftrightarrow \;\;\; [w]_{RH_1} \; := \; \sup_{J\subset \R}m_J \left(\frac{w}{\ave{J}{w}} \log \frac{w}{\ave{J}{w}}\right)\;\;<\;\infty,
\end{equation}
where $\log$ is a regular logarithm base $e$, which could be negative. Nevertheless, by the Jensen inequality $RH_1$ constant defined this way is always nonnegative.

The $RH_1$ constant of the weight $w$ is the natural limit of $RH_p$ constants in the sense that for every $I \subset \R$
\begin{equation}\label{RH1aslimitRHp}
\ave{I}{\left(\frac{w}{\ave{I}{w}} \log \frac{w}{\ave{I}{w}}\right)}\; = \; \lim_{p \rightarrow 1^+} \frac{p}{p-1} \log\; \frac{\ave{I}{(w^p)}^\frac{1}{p}}{\ave{I}{w}}\;
\end{equation}

We want to make one remark about this definition.
\begin{zamech}
The inequality \ref{defRH1} can be rewritten in the following way:
$$
\ave{J}{\left(w\log(w)\right)}\leqslant \ave{J}{w} \; \log (\ave{J}{w})+Q\ave{J}{w}.
$$
Note that since function $x \log x$ is concave, by Jensen's inequality we also have
$$
\ave{J}{w} \; \log (\ave{J}{w}) \leqslant \ave{J}{\left(w\log(w)\right)}.
$$
\end{zamech}

Condition (\ref{defRH1}) is actually much more natural for those places where one is dealing with the Reverse H\"{o}lder conditions rather than with the $A_p$ conditions, see, for example, \cite{Fefferman:89}, \cite{Corporente:2007}, \cite{HytPer2011}.

There is no standard notation here, in some places this class is called $RH_{L\log L}$ since (\ref{defRH1}) is the reverse Jensen's inequality for the function $x \log x$, in other places it is called $G_1$ to emphasize the contribution of Gehring to the study of the Reverse H\"{o}lder classes. Sometimes for the $RH_1$ constant one takes $\sup_{J\subset \R} \exp\left\{m_J \left(\frac{w}{\ave{J}{w}} \log \frac{w}{\ave{J}{w}}\right)\right\}$ to remove logarithm in the right hand side of the (\ref{RH1aslimitRHp}). We keep our notation because it is shorter and its is clear that we are working with the Reverse H\"{o}lder condition.

\paragraph{Different ways to define $RH_1$ constant of the weight $w$.}\label{diffways}
 First, observe that, trivially, logarithm in the definition of the $RH_1$ constant can be replaced by $\log^+ (x)$,
$\left( \log^+ (x) = \max (\log x,0) \right)$ or $\log (e+x)$.
\begin{lemma}

\end{lemma}

Secondly, from the Stein lemma (see \cite{Stein:1969}), we know that
$$
3^{-n} \; \ave{I}{\left( M(f \chi_I) \right)} \; \leqslant \; \ave{I}{\left( f \; \log \left( e + \frac{f}{\ave{I}{f}} \right) \right)} \; \leqslant \; 2^n \; \ave{I}{(f \chi_I)}
$$
Thus an equivalent way to define $RH_1$ constant is
\begin{equation}
\label{def_equiv1_RH1} [w]_{RH_1^\prime} \; := \; \sup \frac{1}{\ave{I}{w}} \int_I M(w \chi_I),
\end{equation}
which, indeed, is one of the ways to define class $A_\infty$, see for example \cite{Wilbook} or \cite{HytPer2011}.

One can also define Reverse H\"{o}lder and $A_\infty$ constants using Luxemburg norms. Same is true for $RH_1$-constant. Let us first define Luxemburg norm of a function in the following way: for an Orlitz function $\Phi: [0,\infty]\mapsto [0,\infty]$, we define $\norm{w}{\Phi(L),I}$ to be:
$$
\norm{w}{\Phi(L),I} \; := \; \inf \left\{ \lambda > 0\colon \; \frac{1}{|I|} \int_I \Phi\left( \frac{|w|}{\lambda} \right) \; \leqslant \; 1 \right\}.
$$
Iwaniec and Verde in \cite{IwVer} showed that for every $w$ and $I\subset \R^n$
$$
\norm{w}{L\log L,I} \; \leqslant \; \int_I \log \left( e + \frac{w}{\ave{I}{w}} \right) dx \; \leqslant \; 2 \norm{w}{L\log L,I},
$$
so another equivalent definition of the $RH_1$ constant of the weight $w$ is
\begin{equation}
\label{def_equiv2_RH1} [w]_{RH_1^{\prime \prime}} \; := \; \sup_{I\subset \R} \frac{\norm{w}{L\log L,I}}{\norm{w}{L,I}}.
\end{equation}

{\bf Comparability of $RH_1$ and $A_\infty$ constants.} Equivalence of the $RH_1$ and $A_\infty$ conditions is known for a long time, but not the relationship between the $RH_1$ and $A_\infty$ constants.
In this paper we prove the following inequality:
\begin{theorem}[Main result 1 : comparability of $RH_1$ and $A_\infty$ constants]
\label{theorem_RH1<Ainfty} A weight $w$ belongs to the Muckenhoupt class $A_\infty$ if and only if $w \in RH_1$. Moreover,
\begin{equation}\label{maininequality}
[w]_{RH_1} \;\leqslant \; C \;  [w]_{A_\infty},
\end{equation}

where the constant $C$ can be taken to be $e$ ($C=e$). Moreover, the constant $C=e$ is the best possible.
\end{theorem}
Bellman function proof of this theorem can be found in Section \ref{s: proof : RH1<eAinfty}. An independent proof of the analogue of this theorem for the constant $[w]_{RH_1^\prime}$ was recently obtained in \cite{HytPer2011}.

Moreover, using a similar Bellman Function approach, one can prove the following theorem.
\begin{theorem}\label{Th:funnybound} If $[w]_{RH_1}=Q$ then
$$
[w]_{\infty}\leqslant C \frac{e^{e^Q-1}}{e^Q},
$$
where $C$ does not depend on $Q$. Moreover, this inequality is sharp in $Q$.
\end{theorem}
We give a sketch of the proof in the Section \ref{Pr:funnybound}. We also note that in the paper \cite{HytPer2011} authors got a bound similar to the Theorem \ref{theorem_RH1<Ainfty} (without sharpness). However, as far as we know the Theorem \ref{Th:funnybound} is new, and we find the bound very surprising.

\paragraph{ 1-Gehring Lemma.} Reverse H\"{o}lder classes have a remarkable self-improvement
property, discovered by Gehring in 1973, see \cite{Gehring:73}.
\begin{theorem}[Gehring's theorem] Suppose $w\in RH_p$ for some $1< p < \infty$.
Then there exists $\varepsilon > 0$, depending only on $p$ and the $RH_p$ constant of $w$, such that $w\in RH_{p+\varepsilon}$.
\end{theorem}

In 1985 Bojarski (see \cite{Boiarsky:1985}) posed the question of finding the sharp dependence of $\varepsilon$ on the $p$ and the $RH_p$ constant of the weight $w$ (and the dimension in multidimensional case). The sharp asymptotic for the case of $RH_p$ constant close to one was obtained by Bojarski (\cite{Boiarsky:1985}) and Wik (\cite{Wik:1992}). In 1990 Sbordone and D'Apuzzo (see \cite{Sbordone:83} and \cite{DapSbor:90}) found sharp dependence for monotone functions and in 1992 Korenovskii (\cite{Korenovskii:1992}) showed that increasing rearrangements do not change the Reverse H\"{o}lder constant of the weight, expanding results of Sbordone and D'Apuzzo to the weights that are not monotone. In 2008 Vasyunin (see \cite{Vasyunin:08}) presented a new proof of the sharp Gehring lemma using method of Bellman functions. All of the above was done for the case $1< p < \infty$ and in dimension one. Let us state the sharp version of the Gehring Lemma.
\begin{theorem}[Sharp Gehring Lemma ($n=1$, $1<p<\infty$)]
\label{theorem_sharp_Gehr_lemma1} Let $w$ be a weight, $w\in RH_p$ for some $p>1$, then $w\in RH_{p+\delta}$ $\forall \delta < \varepsilon$, where $\varepsilon$ is the root of
\begin{equation}
\label{eq_sharp_gehr_lemma1} \frac{1}{p-1}\log \frac{p+\varepsilon - 1}{\varepsilon} - \log \frac{p+\varepsilon}{p+\varepsilon-1} \; = \; \frac{p}{p-1} \log [w]_{RH_p}.
\end{equation}
\end{theorem}

In the following theorem we show that with the $RH_1$ constant defined as above, the Gehring Lemma works for $p=1$ and obtain the sharp dependence of $\varepsilon$ on the $RH_1$ constant of the weight in dimension one.
\begin{theorem}[Main result 2 : Sharp Gehring Lemma ($n=1$, $p=1$)]
\label{theorem_sharp_Gehr_lemma2} Suppose $w\in RH_1$, then $w\in RH_{1+\ep}$, $0<\ep<\varepsilon_-$, where $\varepsilon_-$ is the smallest solution of the equation
\begin{equation}
\label{eq_sharp_gehr_lemma2} \frac{1}{t} - \log \left( \frac{1}{t} + 1 \right) \; = \; [w]_{RH_1}.
\end{equation}
This result is sharp in a sense that for any constant $C$ there exists a weight $w\in RH_1$ with $[w]_{RH_1}=C$ such that $w$ does not belong to $RH_{1+\varepsilon_-}$ with $\varepsilon_-$ defined by \ref{eq_sharp_gehr_lemma2}.
\end{theorem}
Proof of this theorem can be found in Section \ref{s:proof_theorem_Gehr_Lemma_2}.

There are no known extensions of the above sharp results to the higher dimension. The nonsharp dependence of $\varepsilon$ on $p$ and the $RH_p$ constant of the weight $w$ is not hard to trace even in more general case of $\R^n$. Following \cite{Grafakos:03}, \cite{Stein:93} or \cite{GarciaCuervaRubioDeFrancia:85} one can easily show
\begin{equation}
\label{eq_1.8} w\in RH_1 \;\; \Rightarrow \;\; w\in RH_{1+\varepsilon} \;\; with \;\; \varepsilon = \frac{\log 4}{n \log 2 + 8[w]_{RH_1}},
\end{equation}
but this result is far from being sharp. We include the proof of (\ref{eq_1.8}) in Section \ref{s:proof_theorem_Gehr_p=1_n} for completeness.

{\bf 1-Gehring v.s. p-Gehring.} In the end of this section we will show that $p$-Gehring (unfortunately not a sharp one) for any $p>1$ follows from the $1$-Gehring in dimension $n$. We will (except for one step where we use $1$-Gehring Lemma) follow Iwaniec, see \cite{IwVer}.

We start with $p>1$ and $w\in RH_p$, i.e. for any interval $I\subset \R$
$$
\left(\ave{I}{(w^p})\right)^{\frac{1}{p}} \; \leqslant \; [w]_{RH_p} \ave{I}{w}.
$$
We would like to show that $w\in RH_{p+\delta}$ for some $\delta>0$. Trivially, we have pointwise inequality for the Hardy-Littlewood Maximal function $M$
$$
\left( M(w^p) \right)^{\frac{1}{p}} \; \leqslant \; [w]_{RH_p} M(w).
$$
Since by our assumption $w\in L_p(I)$, $M(w)$ is in $L_p(I)$ as well, so by the above inequality $M(w^p) \in L_1$. By the famous result of Stein \cite{Stein:1969} it implies that $w^p \in L\log L(I)$ and
$$
\ave{I}{\left( w^p \log \left( e+\frac{w^p}{\ave{I}{(w^p)}} \right) \right)} \; \leqslant \; 2^n \; \ave{I}{(M(w^p))}
$$
which, by Weiner, 
is bounded from above by
$$
\leqslant \; 2^n [w]_{RH_p}^p 3^n \frac{p}{p-1} 2^p \; \ave{I}{(w^p)}.
$$
So, by the above, $w^p \in RH_1$ with
$$
[w]_{RH_1} = \sup_{I\subset \R} \ave{I}{\left( \frac{w^p}{\ave{I}{(w^p)}} \log \left( e+\frac{w^p}{\ave{I}{(w^p)}} \right) \right)} \leqslant 6^n [w]_{RH_p}^p 2^p \frac{p}{p-1}.
$$
Now all we need is to apply 1-Gehring Lemma, that there exists an $\varepsilon >0$ such that $w^p\in RH_{1+\varepsilon}$, which trivially implies that $w\in RH_{p+\delta}$ with $\delta=p\varepsilon$.

{\bf Some useful technical Lemmas.} We also prove two technical lemmas, which, we think, can be interesting on their own. The first lemma is the $RH_1$ case missing in \cite{RVV:2010}, where the analogues were shown for $RH_p$ and $A_q$ for $1<p<\infty$ and $1<q\leqslant \infty$.

\begin{lemma}\label{lemma_technical1}
Take a function $w\in RH_1$ and define
$$
w_n(t)=\begin{cases} \frac{1}{n}, &w(t)\leqslant \frac{1}{n}\\
w(t), &\frac{1}{n}\leqslant w(t)\leqslant n \\
n, &w(t)\geqslant n

                    \end{cases}.
$$
%

Then
$$
[w_n]_{RH_1}\leqslant [w]_{RH_1}.
$$

Moreover, the same holds for any function $w\in A_\infty$ with replacing $[.]_{RH_1}$ by $[.]_{\infty}$.
\end{lemma}

Next lemma is the $RH_1$ analogue of Vasyunin's lemmas from \cite{Vasyuninrus:04} and \cite{Vasyunin:08}. Lemma \ref{theorem_Vasyunin} is taken from these articles.
\begin{lemma}\label{lemma1-2}
Fix $Q_1>Q>0$ and denote $\Om_{Q_1}=\{(x,y)\colon x\log(x)\leqslant y\leqslant x\log(x)+Q_1 x\}$. Then for every $w\in RH_1$, $[w]_{RH_1}<Q$, there are two intervals $I^{+}$ and $I^{-}$ such that $I=I^{-}\cup I^{+}$ and if $x^{\pm}=\left(\av{w}{I^{\pm}}, \av{(w\log(w))}{I^{\pm}}\right)$ then $[x^{-}, x^{+}]\subset \Om_{Q_1}$.
Also the parameters $\alpha^{\pm}=\frac{|I^{\pm}|}{|I|}$ can be taken separated from $0$ and $1$ uniformly with respect to $w$.
\end{lemma}
\begin{lemma}
\label{theorem_Vasyunin} Fix $Q_1>Q>0$ and denote $\Om_{Q_1}=\{(x,y)\colon 1\leqslant xe^{-y}\leqslant Q_1\}$.
Then for every $w\in A_\infty$, $[w]_{\infty}<Q$, there are two intervals $I^{+}$ and $I^{-}$ such that $I=I^{-}\cup I^{+}$ and if $x^{\pm}=\left(\av{w}{I^{\pm}}, \av{(\log(w))}{I^{\pm}}\right)$ then $[x^{-}, x^{+}]\subset \Om_{Q_1}$.
Also the parameters $\alpha^{\pm}=\frac{|I^{\pm}|}{|I|}$ can be taken separated from $0$ and $1$ uniformly with respect to $w$.
%
\end{lemma}
Proofs of Lemma \ref{lemma_technical1} and Lemma \ref{lemma1-2} are very similar to the proofs of their $RH_p$ analogues from \cite{RVV:2010} and \cite{Vasyunin:08}.

We would like, however, to give a heuristic idea why these lemmata are true.
Fix $Q_1>Q$ and take a weight $w$, such that $[w]_{\infty}\leqslant Q$. First we take intervals $I_{\pm}$, such that $|I_\pm|=\frac{1}{2}|I|$. If the line segment, described above, is in $\Om_{Q_1}$, then we stop. If no, we start enlarging $I_+$. The line segment, which connects $(x_-, y_-)$ and $(x_+, y_+)$ starts turning and finally gets into $\Om_{Q_1}$.

The only detail is that the parameters $\frac{|I_{\pm}|}{|I|}$ can be chosen bounded away from $0$ and $1$, independently on $w$. This is a technical calculation, and we refer the curious reader to the paper \cite{Vasyunin:08}.

\bigskip

{\bf{Acknowledgements}}

Authors are grateful to A. Volberg for useful suggestions in proving Theorem \ref{theorem_RH1<Ainfty} and to V. Vasyunin for useful discussions. We would also like to thank Steve Hofmann and Chema Martell for suggestions for Section \ref{s: PDE}.

We want to thank Sergei Treil for useful discussions about the mathematical part of the paper and about the way to write its several parts.

Finally, we would like to express our gratitude to C. Thiele, I. Uriarte-Tuero and A. Volberg for organizing the Summer School 2010 in UCLA, where this paper was originated and C. P\'{e}rez and R. Esp\'{\i}nola for organizing the Summer School 2011 in Seville, where we finished this paper.

\section{Applications.}
\label{s: applications}

\subsection{Dirichlet problem for elliptic PDE's}
\label{s: PDE}

In this section we will implicitly follow \cite{HoMa:2010}. This is the reason we will work with $\R_+^{n+1}$.

We start with real symmetric second order elliptic operator
\begin{equation}\label{defL}L f (X) : = - div A(X) \nabla f(X) , \;\;\; X \in \R^{n+1}_+,\end{equation}
with $A(X) = \left(a_{i,j}(X)\right)_{1 \leqslant i,j \leqslant n+1}$ being real, symmetric $(n+1) \times (n+1)$ matrix such that $a_{i,j} \in L_\infty (\R^{n+1}_+)$ for $1 \leqslant i,j \leqslant n+1$, and $A$ is uniformly elliptic, that is, there exists $0<\lambda \leqslant 1$ such that
$$
\lambda \mdl{\xi}^2 \leqslant  A(X) \xi \cdot \xi \leqslant \lambda^{-1} \mdl{\xi}^2
$$
for all $\xi \in \R^{n+1}$ and almost every $X \in \R_+^{n+1}$.

If $f$ is a continuous function on $\R^n$, then there exists a unique function $u$, continuous on ${{\R^{n+1}_+}}$, so that $L u = 0$ in $\R^{n+1}_+$ and $u = f$ on $\R^n$. Then for a point $X_0 \in \R^{n+1}_+$ mapping $f \in {\mathcal{C}} (\R^n) \rightarrow u(X_0)$ is a positive linear functional so that there exists a unique nonnegative measure $\omega^{X_0}$ on $\R^n$ such that for every $f \in {\mathcal{C}} (\R^n)$,
$$
\int_{\R^n} f d \omega^{X_0} = u(X_0).
$$
This measure $\omega^{X_0}$ is called the harmonic measure associated to $L$. Let us fix the point $X_0$ and drop the index $\omega = \omega^{X_0}$. It is often important for applications to know whether or not $\omega$ is absolutely continuous with respect to the Lebesgue (surface) measure $dx$ on $\R^n$. If this is the case, it is also of interest to know how nice the Radon-Nikodym derivative $\kappa =\frac{d \omega}{dx}$ (the Poisson kernel) is. It is a well-known fact, that Dirichlet problem for $L$ is solvable in $L_{p^\prime}, \frac{1}{p} + \frac{1}{p^\prime}=1$, if and only if $\kappa \in RH_p$ (for precise statement of the theorem see \cite{HoMa:2010}, \cite{FeffermanKenigPipher:91}, or \cite{Kenig:91}).

According to Caffarelli, Fabbes, and Kenig \cite{CFK:1981} there exist elliptic operators $L$ of form (\ref{defL}) such that the measure $\omega$ associated to $L$ is not absolutely continuous with respect to the Lebesgue measure $dx$. Later, Fabes, Jerison, and Kenig showed in \cite{FJK:1984} that if matrix $A(X) = \left(a_{i,j}(X)\right)_{1 \leqslant i,j \leqslant n+1}$ of our operator $L$ has continuous entries on ${\R^{n+1}_+}$ and the modulus of continuity is good enough,
then $\omega$ is absolutely continuous with respect to the Lebesgue (surface) measure $dx$, and, moreover its Radon-Nikodym derivative $\kappa$ belongs to the Reverse Holder class $RH_2$. Then in \cite{Dahlberg:1986} Dahlberg extended this to the following result for the solvability of $L$ in $L_{p^\prime}$ in the case when $L$ is a small perturbation of a solvable operator $L_0$.
Given two elliptic operators $L_0$ and $L$ as above with associated matrices $A_0$ and $A$, we define their disagreement as
$$
a(X) := \sup_{\mdl{X-Y}_\infty < \rho(X) /2} \mdl{A(Y)-A_0(Y)}.
$$
\begin{theorem}\label{theoremDahlberg} (Dahlberg'86)
Let $L_0$ and $L$ be two operators as above with $a$ being their disagreement, and let $\omega_0$, $\omega$ denote their respective harmonic measures. Assume that the measure $\frac{a(X)^2}{\rho(X)} dX$ is a Carleson measure:
\begin{equation}\label{ACarleson}
\sup_{Q\in \R^n} \frac{1}{\mdl{Q}} \int_{R_Q} \frac{a(X)^2}{\rho(X)} dX < \infty,
\end{equation}
where $R_Q$ is a Carleson box associated to $Q$.

Suppose also that Carleson measure $\frac{a(X)^2}{\rho(X)} dX$ has vanishing trace:
\begin{equation}\label{DahlbergVTcondition}
\lim_{r \rightarrow 0^+} \sup_{Q\in \R^n, \ell(Q) \leqslant r} \frac{1}{\mdl{Q}} \int_{R_Q} \frac{a(X)^2}{\rho(X)} dX  \; = \; 0.
\end{equation}
Then if $\kappa_0 \in RH_p$ for some $1 < p < \infty$ implies $\kappa \in RH_p$, i.e. if $L_0$ is solvable in $L_{p^\prime}$  then $L$ is solvable in $L_{p^\prime}$ as well.
\end{theorem}
In \cite{Fefferman:89} Robert Fefferman showed that in the limiting case $p = 1$ condition (\ref{DahlbergVTcondition}) can be significantly relaxed.

\begin{theorem}\label{theoremRFefferman}(Fefferman'89)
Let $L_0$ and $L$ be two operators as above with $a$ being their disagreement, and let $\omega_0$, $\omega$ denote their respective harmonic measures. Assume that the measure $\frac{a(X)^2}{\rho(X)} dX$ is a Carleson measure (i.e. it satisfies (\ref{ACarleson})).
Suppose also that we have
\begin{equation} \label{conditionRF}
\norm{{\mathcal{A}}(x)}{L_\infty (\R^n)} < \infty\;, \;\;\;\;\; where \;\;\; {\mathcal A}(x) := \left(\int_{\Gamma (x)} \frac{a(X)^2}{\rho(X)^n} dX\right)^\frac{1}{2}
\end{equation}
Then $\kappa_0 \in A_\infty ( = RH_1)$  implies $\kappa \in A_\infty ( = RH_1)$, i.e. if $L_0$ is solvable in $L_{p^\prime}$, $1<p^\prime<\infty$, then $L$ is solvable in $L_{q^\prime}$ for the some $1 < q^\prime < \infty$.

Moreover, $[\kappa]_{RH_1} \leqslant C [\kappa_0]_{RH_1}$ holds with constant $C$ depending on the $L_\infty$ norm of ${\mathcal A}(x)$, the ellipticity constant of the operators $L_0$ and $L$ and the dimension $n$.
\end{theorem}
Robert Fefferman does not state the dependence of the $RH_1$ constants, but it follows from his proof.

In 1991 Fefferman Kenig and Pipher come up with a different method and show that even if condition (\ref{conditionRF}) is omitted, having that the measure $\frac{a(X)^2}{\rho(X)} dX$ is Carleson is enough to keep Radon-Nikodym derivatives in $A_\infty$.

\begin{theorem}\label{theoremFKP} (Fefferman-Kenig-Pipher'91). Let $L_0$ and $L$ be two operators as above with $a$ being their disagreement, and let $\omega_0$, $\omega$ denote their respective harmonic measures. Assume that $\frac{a(X)^2}{\rho(X)} dX$ is a Carleson measure (i.e. it satisfies (\ref{ACarleson})).

Then we have that $\kappa_0 \in A_\infty ( = RH_1)$ implies $\kappa \in A_\infty ( = RH_1)$. More precisely, if $L_0$ is solvable in some $L_{p^\prime}$, $1<p^\prime<\infty$, there exists $1<q^\prime < \infty$ such that $L$ is solvable in $L_{q^\prime}$.
\end{theorem}
This theorem looks like a clear generalization of Fefferman's result, but notice that the relationship between $RH_1$ constants of $\kappa$ and $\kappa_0$ is not traced anymore. In this area people normally do not need estimates on the Reverse H\"{o}lder constants, what matters is the value of $p$, for which $\kappa \in RH_p$. Examples (see \cite{FeffermanKenigPipher:91}) suggest that under conditions of  Theorem \ref{theoremFKP}, weaker than the vanishing trace conditions in Dahlberg's theorem, $p$ will not be preserved (i.e. $\kappa_0\in RH_p$ will not imply that $\kappa \in RH_p$), we can only claim that for a given $p$ such that $\kappa_0 \in RH_p$ there exists  a $q$ such that $\kappa \in RH_q$. The natural question to ask here : Is there anything we can say about $q$?

This is where Fefferman's estimates on the $RH_1$ constant of $\kappa$ turn out to be very handy. When we know $[\kappa]_{RH_1}$ we can use the limiting case of the Gehring's theorem for $p=1$, (\ref{eq_1.8}), in the following way:
\begin{theorem}\label{sharpFKP}  Let $L_0$ and $L$ be two operators as above with $a$ being their disagreement, and let $\omega_0$, $\omega$ denote their respective harmonic measures. Assume that $\frac{a(X)^2}{\rho(X)} dX$ is a Carleson measure (i.e. it satisfies (\ref{ACarleson})).

(1)(Fefferman-Kenig-Pipher) We have that $\omega_0 \in A_\infty (=RH_1)$ implies $\omega \in A_\infty(=RH_1)$. More precisely, if $L_0$ is solvable in some $L_{p^\prime}$, $1<p^\prime<\infty$, there exists $1<q^\prime < \infty$ such that $L$ is solvable in $L_{q^\prime}$.

(2)(R.Fefferman) Suppose in addition to (\ref{ACarleson}) the Fefferman's condition (\ref{conditionRF}) is satisfied.

Then $\kappa_0 \in A_\infty (=RH_1)$  implies $\kappa \in A_\infty (=RH_1)$, and, moreover $[\kappa]_{RH_1} \leqslant C [\kappa_0]_{RH_1}$ with $C = C(\norm{{\mathcal A}(x)}{L_\infty (\R^n)}, \lambda, n)$, which means that
$$
\kappa \in RH_{1+\varepsilon} \;\; with \;\; \varepsilon =  \frac{\log 4}{n \log 2 + 8 C [\kappa_0]_{RH_1}}.
$$

i.e. if $L_0$ is solvable in $L_{p^\prime}$ ($\kappa_0 \in RH_p$), $1<p^\prime<\infty$, then $L$ is solvable in $L_{q^\prime}$ for the $q = 1+ \frac{\log 4}{n \log 2 + 8 C [\kappa_0]_{RH_1}}$. Note also that for any $1<p<\infty$ we have $[\kappa_0]_{RH_1} \leqslant \frac{p}{p-1} \log[\kappa_0]_{RH_p}$.

(3)(Dahlberg) Suppose also that measure $\frac{a(X)^2}{\rho(X)} dX$ has vanishing trace, i.e. satisfies (\ref{DahlbergVTcondition}).

Then if $\kappa_0 \in RH_p$ for some $1 < p < \infty$ implies $\kappa \in RH_p$, i.e. if $L_0$ is solvable in $L_{p^\prime}$, $1<p^\prime<\infty$, then $L$ is solvable in $L_{p^\prime}$ for the same $p$.
\end{theorem}
This theorem (the $\varepsilon$ part of $(2)$) is not sharp. Sharp $1$-Gehring would help the part $(2)$, and we will try to get it, but it would not help to trace the dependence of $q$ on $p$ in part $(1)$. In fact, it is not clear here if Fefferman's assumption can be relaxed.

\section{Proofs.}
\label{s: proofs}
\subsection{Proof of Theorem \ref{theorem_RH1<Ainfty}}
\label{s: proof : RH1<eAinfty}
(Bellman function proof) We will prove that if $w$ belongs to the Muckenhoupt class $A_\infty$ on the interval $J$, $w\in A_\infty(J)$, i.e.
\begin{equation}
\label{ineq_1} \sup_{I\subset J} \ave{I}{w} \; e^{-\ave{I}{(\log w)}} \; \leqslant \; [w]_{A_{\infty,J}},
\end{equation}
then
\begin{equation}
\label{ineq_RH1_Ainf} \ave{J}{(w \log w)} \; \leqslant \; \ave{J}{w} \; \ave{J}{(\log w)} + e [w]_{A_{\infty,J}} \ave{J}{w}.
\end{equation}

We start with the following Lemma:

\begin{lemma}
\label{lemma1} In order to prove inequality (\ref{ineq_RH1_Ainf}), it is enough to show that for every small $\varepsilon > 0$ and a Bellman function $B_{Q,\varepsilon} = B_{Q,\varepsilon}(x,y) = B(x,y)$ (we will drop index $Q$ for simplicity), defined on the domain
$$
\Omega_{Q+\varepsilon} \; = \; \left\{ \vec{x} = (x,y)\in \mathbb{R}^2: x\geqslant 0, \;\;\; 1 \leqslant xe^{-y} \leqslant Q+\varepsilon \right\}
$$
that satisfies the following properties:

(1) $B$ is continuous on $\Omega_{Q+\varepsilon}$

(2) $B(x,y)$ is 
bounded from above by $x\log x + eQx$:
\begin{equation}
 B(x,y) \; \leqslant \; x\log x + eQx \;\;\;\;\; \forall (x,y)\in \Omega_{Q+\varepsilon},
\end{equation}
and
\begin{equation}\label{ineq_3}B(x, y)\geqslant x\log(x).\end{equation}

(3) $B(x,y)$ is locally convex on $\Omega_{Q+\varepsilon}$:
\begin{equation}
\label{ineq_4}
B^{\prime\prime}_{yy}(x,y) \leqslant 0 \;\;\; and \;\;\; \det \left(
\begin{tabular}{c c} $B^{\prime\prime}_{xx}$ & $B^{\prime\prime}_{xy}$\\
$B^{\prime\prime}_{xy}$ & $B^{\prime\prime}_{yy}$
\end{tabular} \right) = 0 \;\;\; \forall (x,y)\in \Omega_{Q+\varepsilon}.
\end{equation}
\end{lemma}
We will first prove Lemma \ref{lemma1} and then present the function $B$, satisfying the above properties.

\begin{proof}
(Proof of Lemma \ref{lemma1}) Let $w$ be an $A_\infty$-weight on the interval $J$. We will first truncate it by $\frac{1}{n}$ from below and by $n$ from above:
\begin{eqnarray}
\label{eq_5} w_n(t) \; := \; \left\{
\begin{tabular}{c c}
$n$, & $w(t) \geqslant n$\\
$w(t)$, & $\frac{1}{n} \leqslant w(t) \leqslant n$\\
$\frac{1}{n}$, & $w(t) \leqslant \frac{1}{n}$
\end{tabular} \right.
\end{eqnarray}
and show that Lemma \ref{lemma1} holds for the weight $w_n$ with all constants independent of $n$. Then by sending $n$ to infinity and applying Lebesgue dominated convergence theorem one obtains the inequality (\ref{ineq_RH1_Ainf}) for any $w\in L_1^{loc}(\mathbb{R})$.

Thus, we consider the truncated weight $w_n(t)$ on the interval $J\in \mathbb{R}$. By the Lemma \ref{lemma_technical1}, we know that the $A_\infty$ constant of the truncated weight $w_n(t)$ does not exceed the $A_\infty$ constant of the original weight $w$.


Now, for every interval $I\subset J\subset \mathbb{R}$, let
$$
\vec{x}_I \; = \; (x_I,y_I) \; := \; \left( \ave{I}{(w_n)}, \ave{I}{(\log w_n)} \right).
$$
Then, for every such $I$, by the Reznikov-Vasyunin-Volberg Theorem, $\vec{x}_I \in \Omega_{[w]_{A_\infty}}$ and, moreover, $\frac{1}{n} \leqslant x_I \leqslant n$.

Next, we use the Lemma \ref{theorem_Vasyunin} in order to construct the sequence $\left\{ I_k^j\right\}_{1\leqslant j \leqslant 2^k, \;k,j\in \mathbb{N}}$ of subintervals of $I$ with properties that $\forall k\in N$  set $ {\mathcal{J}}_k :=\left\{ I_k^j\right\}_{1\leqslant j \leqslant 2^k}$ forms a partition of $J$, lengths of $I_k^j$ approach $0$ as $k\rightarrow \infty$ and for every $k,j\in \mathbb{N}$, $1\leqslant j \leqslant 2^k-1$, the line segment connecting points $\vec{x}_{I_k^j}$ and $\vec{x}_{I_k^{j+1}}$ belongs to the extended domain $\Omega_{[w]_{A_\infty}+\varepsilon}$, while points $\left\{ \vec{x}_{I_k^j} \right\}$ lie in $\Omega_{[w]_{A_\infty}}$.

%

We apply the Lemma \ref{theorem_Vasyunin} to the interval $J =: I_0^1$ with $\varepsilon > 0$ from conditions of Lemma \ref{lemma1} to split it into $J=I_0^1 = I_1^1 \cup I_1^2$. We repeat this procedure with the same $\varepsilon$ for $I_1^1$ and $I_1^2$ and obtain $I_2^1$, $I_2^2$, $I_2^3$ and $I_2^4$. This way we build $\left\{ I_k^j \right\}_{k,j\in \mathbb{N}, \;\; 1\leqslant j \leqslant 2^k}$.

Since both $\delta^k$ and $(1-\delta)^k$ $\rightarrow 0$ as $k\rightarrow \infty$, $\lim_{k\rightarrow \infty} \max_j \mdl{I_k^j} = 0$. By the construction, $\forall k\in \mathbb{N}$ $J= \bigcup_j I_k^j$ and, finally, for every $k,j\in \mathbb{N}$, $1\leqslant j \leqslant 2^k$ we have $\vec{x}_{I_k^j}\in \Omega_{[w]_{A_\infty}}$ and the closed interval $\left[ \vec{x}_{I_k^j}; \vec{x}_{I_k^{j+1}} \right] \subset \Omega_{[w]_{A_\infty}+\varepsilon}$ whenever $I_k^j$ and $I_k^{j+1}$ come from the same parent $I_{k-1}^i$.

Denote
$$
x_{k,n}(s) \; := \; \ave{I_k^j}{(w_n)}, \;\;\;\; s\in I_k^j
$$
$$
y_{k,n}(s) \; := \; \ave{I_k^j}{(\log w_n)}, \;\;\;\; s\in I_k^j.
$$
Both $x_{k,n}$ and $y_{k,n}$ are step functions and for almost every $s$ we have that $\left( x_{k,n}(s), \; y_{k,n}(s) \right) \rightarrow \left( w_n(s), \log w_n(s) \right)$ as $k\rightarrow \infty$.

To finish the proof of Lemma \ref{lemma1}, we observe that by the concavity of function $B$,
\begin{eqnarray*}
B(x_I,y_I) &\geqslant& \frac{\mdl{I_+}}{\mdl{I}} B\left( x_{I_+}, y_{I_+} \right) + \frac{\mdl{I_-}}{\mdl{I}} B\left( x_{I_-}, y_{I_-} \right)\\
&\geqslant& \frac{\mdl{I_+}}{\mdl{I}} \frac{\mdl{I_{++}}}{\mdl{I_+}} B\left( x_{I_{++}}, y_{I_{++}} \right) + \frac{\mdl{I_+}}{\mdl{I}} \frac{\mdl{I_{+-}}}{\mdl{I_+}} B\left( x_{I_{+-}}, y_{I_{+-}} \right)\\
&& + \frac{\mdl{I_-}}{\mdl{I}} \frac{\mdl{I_{-+}}}{\mdl{I_-}} B\left( x_{I_{-+}}, y_{I_{-+}} \right) + \frac{\mdl{I_-}}{\mdl{I}} \frac{\mdl{I_{--}}}{\mdl{I_-}} B\left( x_{I_{--}}, y_{I_{--}} \right)\\
&\geqslant& \ldots\\
&\geqslant& \sum_{k,j\in \mathbb{N}, \;\; 1\leqslant j \leqslant 2^k} \frac{\mdl{I_k^j}}{\mdl{I}} B \left( x_{I_k^j}, y_{I_k^j} \right).
\end{eqnarray*}
Therefore, $B(x_J, y_J) \geqslant \frac{1}{\mdl{J}} \int_J B(x_{k,n}(s), y_{k,n}(s)) ds$.

Since $w_n$ was bounded from above and below, $\frac{1}{n} \leqslant w_n(t) \leqslant n$, points $(x_{k,n}, y_{k,n})$ belong to the compact set $K_{w,n} \subset \mathbb{R}^2$. $B$ is continuous, so it is bounded on $K_w$ and, by the Lebesgue dominated convergence theorem and the boundedness property (\ref{ineq_3}) of $B$, we have
\begin{eqnarray*}
B(x_J, y_J) &\geqslant& \lim_{k\rightarrow \infty} \frac{1}{\mdl{J}} \int_J B(x_{n,k}(s), y_{n,k}(s)) ds\\
&\geqslant& \lim_{k\rightarrow \infty} \frac{1}{\mdl{J}} \int_J x_{n,k}(s) \; \log x_{n,k}(s) ds\\
&=& \frac{1}{\mdl{J}} \int_J w_n(s) \; \log w_n(s) ds \; = \; \ave{J}{(w_n \log w_n)},
\end{eqnarray*}
which, in its turn, implies that
\begin{eqnarray*}
\ave{J}{(w_n \log w_n)} &\leqslant& B(x_J, y_J) \; \leqslant \; x_J \log x_J + eQx_J\\
&=& \ave{J}{w_n} \; \log \ave{J}{w_n} + eQ\; \ave{J}{w_n}.
\end{eqnarray*}

Since this bound does not depend on $n$, we send $n\rightarrow \infty$ and obtain desired inequality for all $A_\infty$-weights $w$.

Proof of Lemma \ref{lemma1} is complete.
\end{proof}

Now we need to show that $B$ with the above properties exists. The following lemma will help us define such function $B(x,y)$.

In fact, for any $Q>1$ we will construct the {\it exact} Bellman function:
$$
\mathcal{B}_{Q}(x,y)=\sup\{ \ave{I}{(w\log(w))}\colon \ave{I}{w}=x, \; \ave{I}{(\log(w))}=y, \; [w]_{\infty}\leqslant Q\}.
$$
We will need some preparation. First, let $\gamma$ be the root of the equation
$$t - \log(t)=1+\log(Q),$$
such that $\gamma<1$. Next, fix a point $(x,y)\in \Om_Q = \{(x,y)\colon 1\leqslant xe^{-y}\leqslant Q\}$ and let $v=v(x,y)$ be a root of the equation
$$
y=\frac{\gamma\cdot x}{v} + \log(v)-\gamma.
$$
such that $v\leqslant x$.

In fact, the last equation is an equation of a line $\ell$, such that $(v, \log(v))\in \ell$ and $\ell$ is tangent to the curve $xe^{-y}=Q$. So basically we take a point $(x,y)$ and a tangent line, which passes through this point and goes to the right. This line ``hits'' the curve $xe^{-y}=1$ exactly at the point $(v, \log(v))$.

We are ready to state the following lemma.
\begin{lemma}\label{lemma2}
Let $\gamma$ be as above and $v=v(x,y)$ be a function, implicitly defined (on the domain $\Om_Q$) by the equation
$$
y=\frac{\gamma\cdot x}{v} + \log(v)-\gamma,
$$
and such that $v\leqslant x$.

Denote
$$
B(x,y)=x\log(v)+\frac{x-v}{\gamma}.
$$
Then $B(x,y)$ satisfies all properties from the Lemma \ref{lemma1}.


\end{lemma}
\begin{zamech*}
We remark that instead of writing $Q+\ep$ of $Q_1$ we write $Q$. It is fine because we do it for every $Q>1$.
\end{zamech*}
\begin{proof}
%
%




We leave the differentiation of the function $B$ to the reader. However, we state the answer for several derivatives.
First of all,
\begin{align*}
v'_x = \frac{\gamma v}{\gamma x - v}, \qquad \qquad v'_y = \frac{v^2}{v-\gamma x}.
\end{align*}
Next,
\begin{align*}
&B'_x = \log(v)+\frac{1}{\gamma}+1, \; \; \; \; \; \; B''_{xx} = \frac{\gamma}{\gamma x - v}; \\
&B'_y = -\frac{v}{\gamma}, \; \; \; \; \; \;\; \; \; \; \; \; \;\;\;\;\;\;\;\;\;\; B''_{yy} = -\frac{1}{\gamma}\frac{v^2}{v-\gamma x};\\
&B''_{xy} = \frac{v}{v-\gamma x}.
\end{align*}
Finally, by the definition of $v$, we have $\gamma x \leqslant v$, so $B''_{xx}\leqslant 0$.
We also notice that $B(v, \log(v))=v\log(v)$. Thus, we need to prove that $x\log(x)\leqslant B(x, y)\leqslant x\log(x)+eQx$. We notice that
$$
\frac{B(x,y)-x\log(x)}{x} = \log\frac{v}{x} + \frac{1-\frac{v}{x}}{\gamma}.
$$
Denote $s=\frac{v}{x}$ and notice that $s\in [\gamma, 1]$. Then
$$
\frac{B(x,y)-x\log(x)}{x} = \log(s)+\frac{1-s}{\gamma} = \vf(s).
$$
Since $\vf'(s) = \frac{1}{s}-\frac{1}{\gamma} \leqslant 0$, we get
$$
\frac{B(x,y)-x\log(x)}{x}\leqslant \vf(\gamma) = \log(\gamma)+\frac{1}{\gamma}-1.
$$

It is not hard to check that the last expression is not bigger than $eQ$. Moreover,
$$
\lim\limits_{Q\to \infty}\frac{\log(\gamma)+\frac{1}{\gamma}-1}{Q} = e,
$$
so the constant $e$ is sharp.
Finally,
$$
\frac{B(x,y)-x\log(x)}{x}\geqslant \vf(1) = 0,
$$
so $B(x,y)\geqslant x\log(x)$, which finishes the proof.
\end{proof}

We have proved that our function $B(x,y)$ is bigger or equal than the exact Bellman function $\mathcal{B}(x,y)$. This proves the inequality \eqref{maininequality} with constant $C=e$. In order to prove that this is the best possible constant we need to do the following: for every point $(x,y)\in \Omega_Q$ present a weight $w$, such that $[w]_{\infty}\leqslant Q$, $\ave{I}{w}=x$, $\ave{I}{(\log(w))}=y$, and $B(x,y)=\ave{I}{(w\log(w))}$. The following lemma takes care of this issue.
\begin{lemma}
For a point $(x,y)\in \Om_Q$ consider a function
$$
w(t)=\begin{cases} v \left(\frac{t}{a}\right)^{\gamma-1}, &t\in [0, a] \\
                    v, &t\in [a,1]\end{cases},
$$
where $a$ is taken such that $(x,y)=(\ave{I}{w}, \ave{I}{(\log(w))})$. Then $[w]_{\infty}\leqslant Q$ and $B(x,y)=\ave{I}{(w\log(w))}$.
\end{lemma}
This lemma is technical and we skip the proof. Later we prove a similar Lemma \ref{podpirprimer}. 




\subsection{Proof of Theorem \ref{theorem_sharp_Gehr_lemma2}}
\label{s:proof_theorem_Gehr_Lemma_2}

\begin{proof}
Proof of this theorem is basically the same as the proof of Theorem \ref{s: proof : RH1<eAinfty}. However, we give a detailed proof.
Fix an interval $J \subset \R$. We define the $RH_1$ constant of the weight $w$ on the interval $J$ to be
\begin{equation}\label{d: RH1J}
[w]_{RH_{1,J}}:= \sup_{I\subset J} \ave{I}{\left(\frac{w}{\ave{I}{w}}\log\frac{w}{\ave{I}{w}}\right)} \;<\; \infty
\end{equation}
or, equivalently,
\begin{equation}\label{d: RH1Jalt}
\forall I \subset J \;\;\; \ave{I}{\left(w \log w\right)} \leqslant \ave{I}{w} \log \ave{I}{w} + [w]_{RH_{1,J}} \ave{I}{w}
\end{equation}

For a given $Q>0$ we will show that if $[w]_{RH_{1,J}} \leqslant Q$, then for every $0< \ep <\frac{1}{\gamma_+-1}$, where $\gamma_+$ is the larger solution of equation
$$
\gamma-\log(\gamma)=Q+1,
$$
weight $w$ satisfies the Reverse H\"{o}lder inequality with exponent $1+\ep$ on the interval $J$:
$$
\left(\ave{J}{w^{1+\ep}}\right)^{\frac{1}{1+\ep}}\leqslant C \ave{J}{w}.
$$
\begin{zamech}
We remark that $\ep_-$, defined in \eqref{eq_sharp_gehr_lemma2}, is equal to $\frac{1}{\gamma_+ - 1}$. This is because
$$
\frac{1}{\ep_-} - \log\left(1+\frac{1}{\ep_-}\right) = Q,
$$
so
$$
\left(1+\frac{1}{\ep_-}\right) - \log\left(1+\frac{1}{\ep_-}\right) = Q+1,
$$
and now it is clear that $\gamma_+ = 1+\frac{1}{\ep_-}$.
\end{zamech}
For the given $J  \subset \R$ and $Q>0$ we introduce the following function $\mathcal{B}(x,y)$:
$$
\mathcal{B}(x,y)=\sup\{ \ave{J}{w^{1+\ep}}\colon \ave{J}{w}=x, \; \ave{J}{w\log(w)}=y, \; [w]_{RH_1}\leqslant Q \}.
$$
Note that for every subinterval $I \subset J$ and any weight $w$ satisfying $[w]_{RH_{1,J}}\leqslant Q$, the pair of points $(x_I,y_I): =(\ave{I}{w},\ave{I}{w\log w})$ should lie in the domain
$$
\Omega=\Omega_Q = \{(x,y)\colon x\log x\leqslant y\leqslant x\log x+Qx\}.
$$
Boundary curves of $\Omega$ will be denoted by $\Gamma$ and $\Gamma_Q$:
$$
\Gamma=\{(x,y)\colon x\log x=y\},
$$
$$
\Gamma_Q=\{(x,y)\colon y=x\log x+Qx\}.
$$
So, for every weight $w\in RH_{1,J}$ and every subinterval $I\subset J$, point $(x_I,y_I)$ lies in the domain $\Omega$.  It is not hard to see that the opposite is true as well, for every point $(x,y)\in \Omega$ there is a function $w$, satisfying all properties from the definition of $\mathcal{B}(x,y)$ and such that $(x,y)=(\ave{J}{w},\ave{J}{w \log w})$. In fact, if $(x,y)\in \Omega$ then there are two points $V=(v, v\log v)$ and $U=(u, u\log u)$ on $\Gamma$, such that point $(x,y)$ belongs to the line segment connecting $V$ and $U$, $x=sv+(1-s)u$, $y=sv\log v + (1-s)u\log u$ with $s\in[0,1]$, and the whole interval $[U,V]$ lies inside the domain $\Omega$. To see the existence of $w$, simply observe that for $J=[0,1]$ the weight
$$
w(t)=\begin{cases} v, &t\in [0,s]\\
                    u, &t\in [s,1]
                    \end{cases}
$$
has the above properties. Indeed,
$$
\ave{J}{w}=sv+(1-s)u=x, \;\;\;\;\; \ave{J}{w\log(w)}= s \; v \log v + (1-s) \;  u \log u = y,
$$
and for every interval $I\subset J= [0,1]$ we get that the point $(\ave{I}{w}, \ave{I}{w\log(w)})$ is a convex combination of $V$ and $U$ and, moreover, $[w]_{RH_1}\leqslant Q$ since the line segment $[U,V]$ is inside $\Omega$. A simple rescaling argument proves it for general $J$.
Therefore, $\Omega$ is indeed the domain of $\mathcal{B}$.
\end{proof}

\subsubsection{Geometry of $\Omega$}

We need some basic facts about the geometry of $\Omega$. Namely, we want to investigate the following: if $(x,y)\in \Omega$, then what are the equations of tangents to $\Gamma_Q$, which pass through $(x,y)$? In particular, what happens if $y=x\log x$.
\begin{lemma}
Let $V=(v, v\log v)\in \Gamma$ and $a_v = \gamma_+ v$. Then the line
$$
\ell_v \colon y=(\log v + \gamma_+)x - v\gamma_+
$$
is tangent to $\Gamma_Q$. Moreover, $\ell_v \cap \Gamma_Q = \{(a_v, a_v\log(a_v)+a_v Q)\}$, where $a_v=\gamma_+ v$.
\end{lemma}
Proof of this lemma is a simple exercise in calculus, so we will leave it to the reader.
For any point $(x,y)\in \Omega$ we have a line $\ell(x,y)$ now, tangent to $\Gamma_Q$, which passes through $(x,y)$ and has an equation
$$
y=(\log v + \gamma_+)x - v\gamma_+,
$$
where $v\leqslant x$.

Take $a_v=\gamma_+ v$, so that we have $v\leqslant x\leqslant a_v$.
\begin{center}
\includegraphics[width=0.5\linewidth]{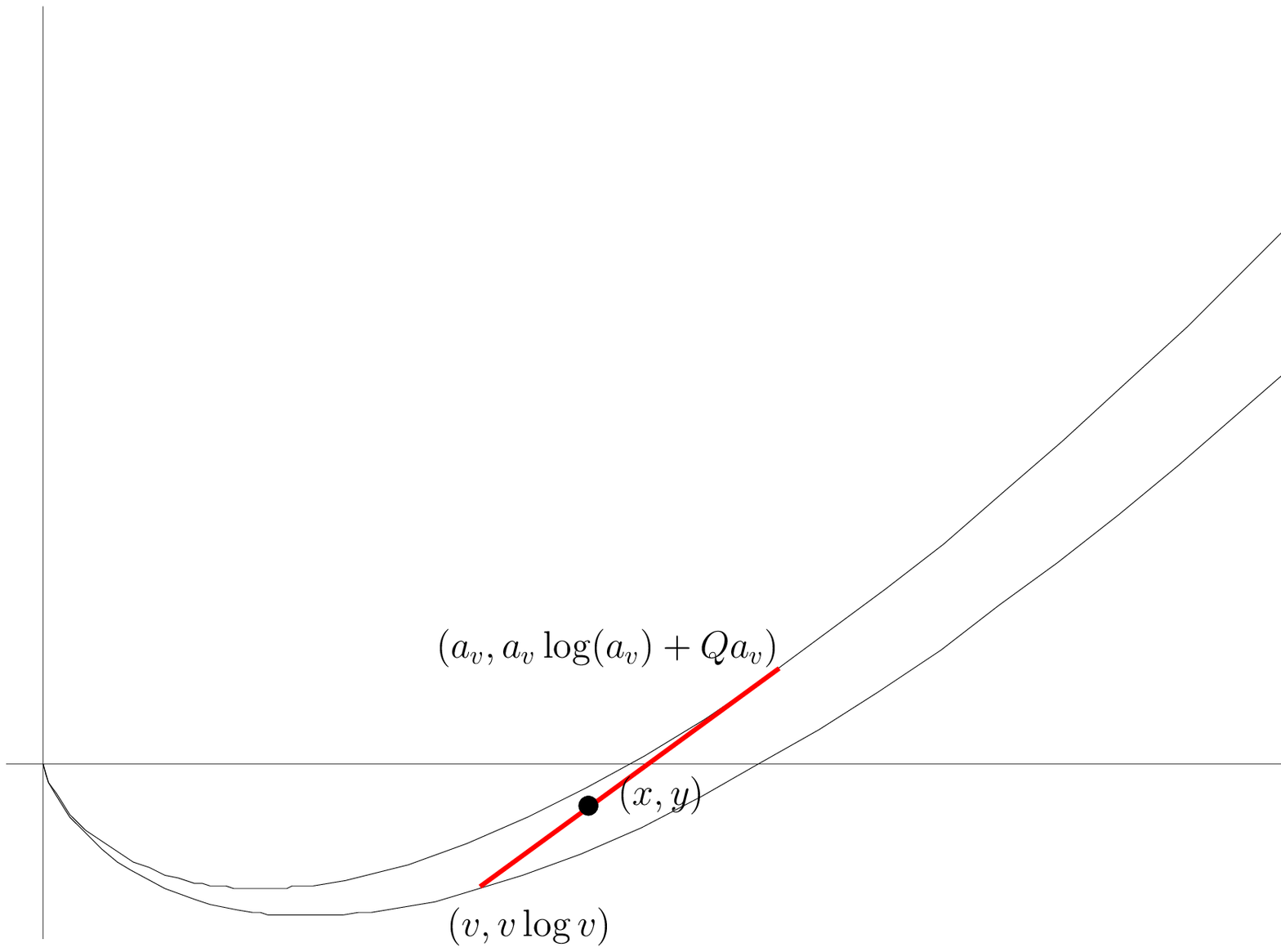}
\end{center}
\begin{center}
\includegraphics[width=0.5\linewidth]{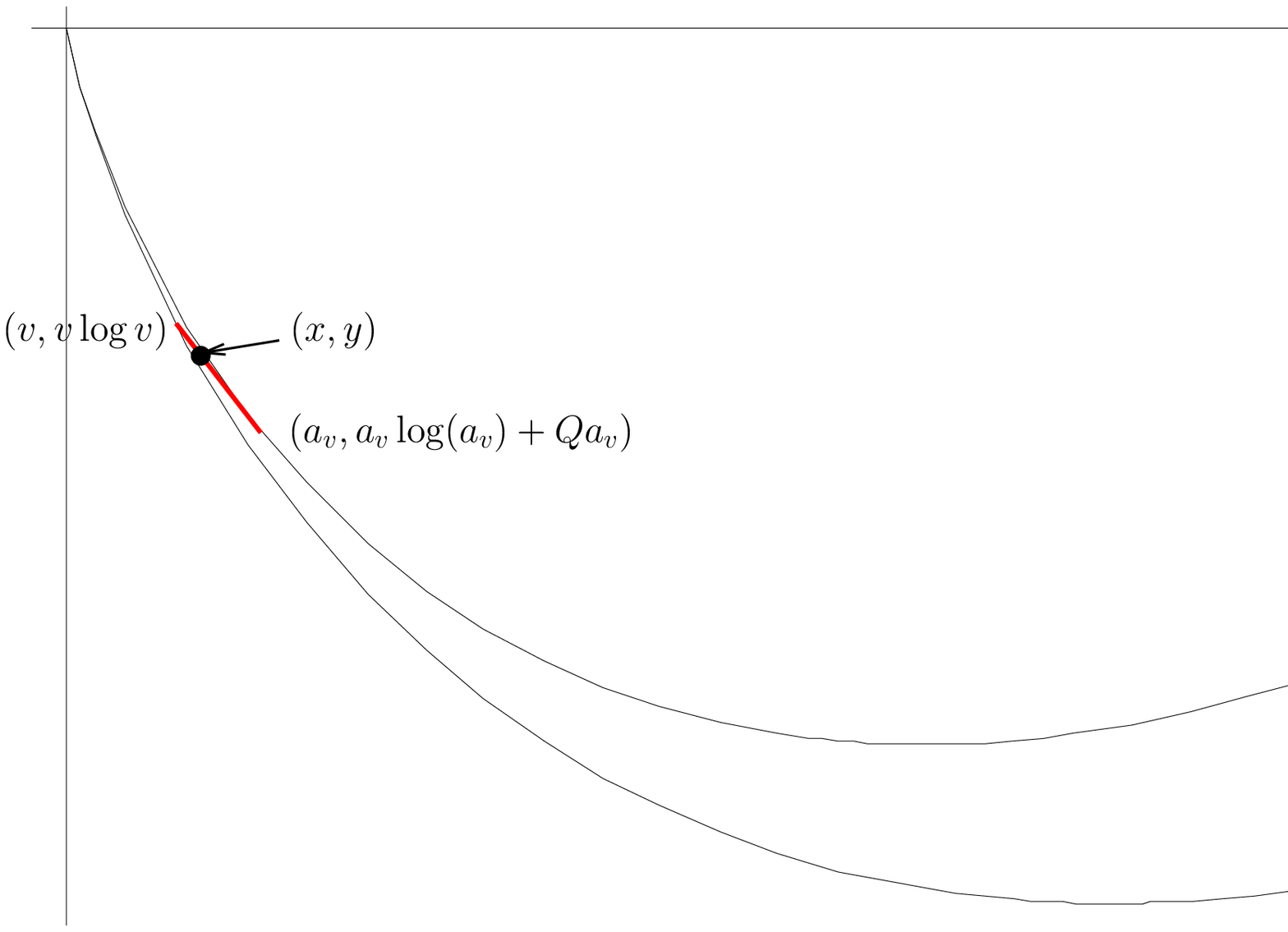}
\end{center}
Now we are ready to formulate the following theorem.
\begin{theorem}\label{theorem_Bellman}
Assume $0<\ep<\frac{1}{\gamma_+-1}$. Then
$$
\mathcal{B}(x,y)=\frac{v(x,y)^\ep}{1+\ep-\gamma\ep}(x(1+\ep)-\ep\gamma v(x,y)),
$$
where $v(x,y)$ satisfies an implicit formula
$$
y=(\log v + \gamma_+)x - v\gamma_+,
$$
$$v(x,y)\leqslant x \leqslant \gamma_+ v(x,y).$$
Moreover, if $(x,y)\in \Gamma_Q$ then the supremum is attained on a function
$
w_{ex}(t)=\frac{x}{\gamma_+}t^{\frac{1-\gamma_+}{\gamma_+}}
$, while
for an arbitrary $(x,y)\in \Omega$ the supremum is attained on a function of the form
$$
w_{ex}(t)=\begin{cases}Ct^{\frac{1-\gamma_+}{\gamma_+}}, &t\in (0, a]\\
                       Ca^{\frac{1-\gamma_+}{\gamma_+}}, &t\in [a, 1].
                       \end{cases}
$$
\end{theorem}
\begin{proof}
We denote
$$
B(x,y)=\frac{v(x,y)^\ep}{1+\ep-\gamma\ep}(x(1+\ep)-\ep\gamma v(x,y)).
$$
The goal is, therefore, to show that $\mathcal{B}=B$.

We break the proof into several lemmas.
\begin{lemma}
The function $B(x,y)$ is locally concave in $\Omega$. That is, the hessian of $B$
$$
\begin{pmatrix} B''_{xx} & B''_{xy}\\
                B''_{xy} & B''_{yy}
                \end{pmatrix}
$$
is a negatively semidefinite matrix.
\end{lemma}
Checking this condition requires nothing but careful differentiation. However, we want to point out that this lemma is true since $\gamma v(x,y) - x\geqslant 0$ for every $(x,y)\in \Omega$. It shows that we could not consider another tangent line from $(x,y)$ to $\Gamma_Q$.

\begin{flushright}
\begin{flushright}
\begin{flushright}
\begin{flushright}

\end{flushright}
\end{flushright}
\end{flushright}
\end{flushright}

Local concavity of the function $B$ implies the following lemma.
\begin{lemma}
The following inequality holds
$$
B(x,y)\geqslant \mathcal{B}(x,y)
$$
\end{lemma}
\begin{proof}
We first observe that on the boundary curve $\Gamma$ we have $\mathcal{B}(v, v\log v)=v^{1+\ep}=B(v, v\log v)$, since the only admissible function $w$ for the point $(v, v\log v)$ is the constant function $w(t)\equiv v$.

We consider a function $B_{Q_1}$, which is defined like $B$, but with $Q_1$ instead of $Q$.
Take a point $(x,y)$ and an arbitrary $w$, $[w]_{RH_1}\leqslant Q$, such that $(x,y)=(\ave{J}{w}, \ave{J}{w\log(w)})$. Assume that $\frac{1}{n}\leqslant w(t)\leqslant n$ for every $t$. Then, in particular,
$$
\ave{J}{w}\in \left[\frac{1}{n}, n\right].
$$
Therefore, the set $\Upsilon=\{(\ave{I}{w}, \ave{I}{w\log w})\colon I\subset J= [0,1]\}$ is compact. Therefore, $B_{Q_1}$ is bounded on $\Upsilon$. Take now $I^{\pm}$ from the Lemma \ref{lemma1-2}.

By $D_{n}$ we denote the set of intervals of $n$-th generation. For example, $D_{0}=\{I\}$ and $D_{1}=\{ I^{-}, I^{+}\}$. For every interval $J\in D_{n}$ we denote
$$
x^{J}=(\ave{J}{w}, \ave{J}{w\log(w)}).
$$

Since $B_{Q_1}$ is locally concave, we can write
$$
B_{Q_1}(x,y)\geqslant |I^+|B_{Q_1}(x^+) + |I^-|B_{Q_1}(x^-).
$$
Repeating this procedure, we get
$$
B_{Q_1}(x,y)\geqslant \sli_{J\in D_{n}}|J|B_{Q_1}(x^{J})=\ili_{0}^{1} B_{Q_1}(x^{n}(t))dt,
$$
where
$x^{n}(t)$ is a step-function, defined in the following way: take $J\in D_{n}$ and denote $x^{n}(t)=x^{J}, \; t\in J$.
By the Lebesgue differentiation theorem, $x^{n}(t)\to (w(t), w(t)\log w(t))$ for a.e. $t$. Moreover, since $B$ is bounded on the set $\{x^J\}$, we can pass to the limit under integral. We get
$$
B_{Q_1}(x,y)\geqslant \ili_{0}^{1} B_{Q_1}(w(t), w(t)\log(w(t)))dt = \ili_{0}^{1} w^{1+\ep}(t)=\ave{J}{w^{1+\ep}}.
$$

If $w$ is unbounded, we consider
$$
w_n(t)=\begin{cases} \frac{1}{n}, &w(t)\leqslant \frac{1}{n},\\
                    w(t), &w(t)\in [\frac{1}{n}, n],\\
                    n, &w(t)\geqslant n. \end{cases}
$$
Then, by the Lemma \ref{lemma_technical1}, $[w_n]_{RH_1}\leqslant [w]_{RH_1}\leqslant Q$, and
$$
B_{Q_1}(x,y)\; \geqslant \; \ave{J}{w_n^{1+\ep}}.
$$
Using the Lebesgue Monotonic convergence theorem, we get
$$
B_{Q_1}(x,y)\; \geqslant \; \ave{J}{w^{1+\ep}}
$$
for every admissible function $w$. After taking the supremum over $w$ we have
$$
B_{Q_1}(x,y)\geqslant \mathcal{B}(x,y)
$$
for every $Q_1>Q$. Since $B_{Q_1}$ is continuous in $Q$, we can write
$$
B(x,y)\geqslant \mathcal{B}(x,y).
$$
\end{proof}

We have shown that $B(x,y)\geqslant \mathcal{B}(x,y)$.
In order to complete the proof of the theorem, we need to show the opposite inequality,
$$
B(x,y)\leqslant \mathcal{B}(x,y).
$$
\begin{lemma}\label{podpirprimer0}
For every point $(x,y)\in \Om$ there exists a function $w_{ex}$, such that
\begin{align*}
&\ave{I}{w_{ex}}=x, \\
&\ave{I}{(w_{ex} \log(w_{ex}))}=y, \\
&[w_{ex}]_{RH_1}\leqslant Q, \\
& B(x,y)=\ave{I}{(w_{ex}^{1+\ep})}.
\end{align*}
Consequently, $B(x,y)\leqslant \mathcal{B}(x,y)$.
\end{lemma}
First we consider the point $(x,y)\in \Gamma_Q$ and
$$
w_{ex}(t)=\frac{x}{\gamma_+}t^{\frac{1-\gamma_+}{\gamma_+}}.
$$
\begin{lemma}\label{podpirprimer}
The function $w_{ex}$ satisfies the inequality $[w_{ex}]_{RH_1}\leqslant Q$.
Moreover, for every $\ep<\frac{1}{\gamma_+ - 1}$ we have
$$B(\ave{I}{w_{ex}}, \ave{I}{(w_{ex}\log(w_{ex}))}) = \ave{I}{(w_{ex}^{1+\ep})}.$$
Finally,
$$\ave{I}{(w_{ex}^{1+\frac{1}{\gamma_+ -1}})} = \infty.$$
\end{lemma}
\begin{zamech}
Notice that the big part of this lemma repeats conditions from the Lemma \ref{podpirprimer0}. However, the last equality shows the sharpness, declared in the Theorem \ref{theorem_sharp_Gehr_lemma2}.
\end{zamech}
\begin{proof}[Proof of the Lemma \ref{podpirprimer}]
To prove that $[w_{ex}]_{RH_1}\leqslant Q$, we take an interval $J=[a,b]$ and write
$$
\av{w}{[a,b]}=x\frac{1}{b-a} (b^{\frac{1}{\gamma_+}}-a^{\frac{1}{\gamma_+}}),
$$
$$
\av{w\log(w)}{[a,b]}=x\log\left(\frac{x}{\gamma_+}\right)(b^{\frac{1}{\gp}}-a^{\frac{1}{\gp}})+x\frac{1-\gp}{\gp}(b^{\frac{1}{\gp}}\log b-a^{\frac{1}{\gp}}\log a)-x(1-\gp)(b^{\frac{1}{\gp}}-a^{\frac{1}{\gp}}).
$$
We substitute
\begin{align*}
\alpha=a^{\frac{1}{\gp}}, & \beta=b^{\frac{1}{\gp}} \\
\alpha=s\beta.
\end{align*}
Then, after some technical calculations, using the definition of $\gp$, we obtain that
$$
\av{w\log(w)}{J}-\av{w}{J}\log(\av{w}{J})-Q\av{w}{J}
$$
has the same sign as
$$
(\gp-1)s\log s - (1-s)\log\frac{1-s}{1-s^{\gp}}.
$$
We now use the following trick. Fix $s\in (0,1)$ and denote
$$
\vf(\gamma)=(\gamma-1)s\log s - (1-s)\log\frac{1-s}{1-s^{\gamma}}.
$$
Obviously, $\vf(1)=0$. Simple calculation shows that $\vf^{\prime}(\gamma)\leqslant 0$ if $\gamma\geqslant 1$, which yields, since $\gp>1$,
$$
\vf(\gp)\leqslant 0.
$$
Therefore, if $s\in (0,1)$, then
$$
(\gp-1)s\log s - (1-s)\log\frac{1-s}{1-s^{\gp}}\leqslant 0.
$$
It is easy to see that the same inequality holds for $s=0$ and $s=1$. Therefore,
$$
\av{w\log(w)}{J}-\av{w}{J}\log(\av{w}{J})-Q\av{w}{J}\leqslant 0,
$$
so $[w_{ex}]_{RH_1}\leqslant Q$.
Moreover,
$$
\ave{J}{w_{ex}^{1+\ep}}=\frac{x^{1+\ep}}{\gp^{\ep}}\frac{1}{1+\ep-\gp \ep}=B(x, y),
$$
since $(x,y)\in \Gamma_Q$, and so $x=\gp v$.
Finally, it is clear that $\ave{I}{(w_{ex}^{1+\frac{1}{\gamma_+-1}})}=\infty$, which finishes the proof of the Lemma \ref{podpirprimer}.
\end{proof}
We now proceed to the arbitrary $(x,y)\in \Om_Q$. Take the tangent $\ell(x,y)$ and defined before $v=v(x,y)$, $a_v=\gp v(x,y)$.
Define
$$
w_{ex}(t)=\begin{cases} v \left(\frac{t}{u}\right)^{\frac{1-\gp}{\gp}}, &t\in[0,u]\\
                    v, &t\in [u,1] \end{cases}.
$$
Note that we ``glue'' two functions: the extremal function for the point $(v, v\log v)$ and the extremal function for $(a, a\log a + Qa)$. We should glue them so $x=\ave{I}{w}$. Since
$$
x=v\frac{a-x}{a-v} + a \frac{x-v}{a-v},
$$
we take $u=\frac{x-v}{a-v}$.
The inequality $[w]_{RH_1}\leqslant Q$ is left to the reader. However, it is a big pleasure to point out that the calculations are not needed because of the proof of such facts (a ``maximizer'' for Bellman function has the desired constant), given by  P. Ivanishvili, N. Osipov, D. Stolyarov, V. Vasyunin and P. Zatickiy, see \cite{BMO:mnogo}.

It remains to show that $\ave{I}{(w_{ex}^{1+\ep})}=B(x,y)$, but it follows from the fact that $B$ is linear on tangent lines to $\Gamma_Q$. We have proved that
$$
B(x,y)=\ave{J}{w_{ex}^{1+\ep}}\leqslant \mathcal{B}(x,y),
$$
which finishes the proof of the Theorem \ref{theorem_Bellman} and the proof of the equality $B(x,y)=\mathcal{B}(x,y)$.

\end{proof}





\subsection{Proof of the Gehring theorem for the case $p=1$ in dimension~$n$}
\label{s:proof_theorem_Gehr_p=1_n}

\begin{theorem}
$w\in RH_1$, then
$$
w\in RH_{1+\ep} \;\;\;\; for \;\; all \;\; \ep \leqslant \frac{\log 4}{n \log 2 + 8[w]_{RH_1^{\prime\prime}}},
$$
where $n$ is the dimension of the underlying space (or related to the doubling constant of the underlying measure).
\end{theorem}
\begin{proof}
Let $w\in RH_1$, then (by (\ref{def_equiv2_RH1}))
\begin{equation}
\label{eq_proof_gehr_11}\forall I\subset \R^n \;\;\;\; \norm{w}{L \log L,I} \; \leqslant \; [w]_{RH_1^{\prime\prime}} \norm{w}{L,I},
\end{equation}
where $\norm{w}{\Phi(L)}$ is Orlitz norm of $w$,
\begin{equation}
\label{eq_proof_gehr_12} if \;\; \Phi = L \log L \;\; then \;\; \bar{\Phi}_n(t) \cong e^L - 1
\end{equation}
and one can write generalized H\"{o}lder's inequality:
$$
\forall I \;\;\;\; \int_I \mdl{f(x) \; g(x)} dx \; \leqslant \; \norm{f}{L\log L,I} \norm{g}{e^L - 1,I}.
$$
Applying this to the $\int_I \mdl{f} w dx$, we can write $\forall f$, $\forall w\in RH_1$ and $\forall I\in D$
\begin{equation}
\label{eq_proof_gehr_14} \int_I |f| w \leqslant 2 \norm{w}{L\log L,I} \norm{f}{e^L-1} \leqslant [by \; (\ref{eq_proof_gehr_11})] \leqslant 2[w]_{RH_1^{\prime\prime}} \norm{w}{L,I} \norm{f}{e^L-1,I}.
\end{equation}
Note first that $\norm{w}{L,I} = \frac{1}{|I|}\int_I w$:
$$
\norm{w}{L,I} \; = \; \inf \left\{ \lambda > 0: \;\; \frac{1}{|I|} \int_I \frac{w}{\lambda} \; \leqslant 1 \right\} \; = \; \frac{1}{|I|}\int_I w.
$$
So, (\ref{eq_proof_gehr_14}) becomes
\begin{equation}
\label{eq_proof_gehr_15} \frac{1}{|I|}\int_I |f|w \; \leqslant \; 2 [w]_{RH_1^{\prime\prime}} \frac{1}{|I|}\int_I w \norm{f}{e^L-1,I}.
\end{equation}
In order to apply inequality (\ref{eq_proof_gehr_15}) to the $f=\chi_E$ for an $E\subset I$,
\begin{eqnarray*}
\norm{\chi_E}{e^L-1,I} &=& \inf \left\{ \lambda > 0: \;\; \frac{1}{|I|} \int_I e^{\frac{\chi_E}{\lambda}} - 1 \leqslant 1 \right\}\\
&=& \inf \left\{ \lambda > 0: \;\; \frac{1}{|I|} \int_E e^{\frac{1}{\lambda}} - 1 \leqslant 1 \right\}\\
&=& \inf \left\{ \lambda > 0: \;\; \frac{|E|}{|I|} \left( e^{\frac{1}{\lambda}} - 1 \right) \leqslant 1 \right\} = \frac{1}{\log \left( 1+\frac{|I|}{|E|} \right)}
\end{eqnarray*}
and then (\ref{eq_proof_gehr_15}) applied to $f=\chi_E$ implies:
\begin{equation}
\label{eq_proof_gehr_17} \frac{w(E)}{w(I)} \; \leqslant \; 2[w]_{RH_1^{\prime\prime}} \frac{1}{\log \left( 1+ \frac{|I|}{|E|} \right)}
\end{equation}
or
$$
\frac{w(E)}{w(I)} \log \left( 1+ \frac{|I|}{|E|} \right) \; \leqslant \; 2[w]_{RH_1^{\prime\prime}}.
$$
Note also that, since $\frac{|I|}{|E|} \geqslant 1$, $\log \left( 1+ \frac{|I|}{|E|} \right) \leqslant \log \left( 2\frac{|I|}{|E|} \right)$.
Take $$\alpha = \frac{1}{e^{8[w]_{RH_1^{\prime\prime}}}-1},$$ then whenever $\frac{|E|}{|I|} \leqslant \alpha$, since $\log \left( 1+\frac{1}{\alpha} \right)$ is a decreasing function of $\alpha$,
\begin{eqnarray*}
\log \left( 1+ \frac{|I|}{|E|} \right) &=& \log \left( 1+ \frac{1}{\frac{|E|}{|I|}} \right) \geqslant \log \left( 1+ \frac{1}{\alpha} \right)=\\
&=& \log \left( e^{8[w]_{RH_1^{\prime\prime}}} \right) = 8[w]_{RH_1^{\prime\prime}},
\end{eqnarray*}
i.e. whenever $\frac{|E|}{|I|}\leqslant \alpha = \frac{1}{e^{8[w]_{RH_1^{\prime\prime}}}-1}$, we get $\log \left( 1+ \frac{|I|}{|E|} \right) \geqslant 8[w]_{RH_1^{\prime\prime}}$, so we can write (\ref{eq_proof_gehr_17}) as
$$
\frac{w(E)}{w(I)} \;\leqslant\; 2[w]_{RH_1^{\prime\prime}} \frac{1}{\log \left( 1+ \frac{|I|}{|E|} \right)} \; \leqslant \; \frac{2[w]_{RH_1^{\prime\prime}}}{8[w]_{RH_1^{\prime\prime}}} = \frac{1}{4},
$$
or, for simplicity,
\begin{equation}
\label{eq_proof_gehr_20} \frac{|E|}{|I|} \leqslant \frac{1}{e^{8[w]_{RH_1^{\prime\prime}}}-1} = \alpha \;\;\; \Rightarrow \;\;\; \frac{w(E)}{w(I)} \leqslant \frac{1}{4} =: \beta.
\end{equation}
See Rubio de Francia-Garcia-Cuerva book, page 398, in order for $w$ to belong to $RH_{1+\varepsilon}$, it is enough to pick an $\varepsilon$ such that $(2^n \alpha^{-1})^{\varepsilon}\beta <1$. For our choice of $\alpha$ and $\beta$ in \eqref{eq_proof_gehr_20} we need to solve for $\varepsilon$ the following inequality:
$$
\left(2^n \left(e^{8[w]_{RH_1^{\prime\prime}}}-1\right)\right)^{\varepsilon} \frac{1}{4}<1.
$$
To satisfy this inequality it is enough to choose an $\varepsilon$ such that
$$
\left(2^n e^{8[w]_{RH_1^{\prime\prime}}}\right)^{\varepsilon}=4,
$$
which yields to
$$
\varepsilon = \frac{\log(4)}{n\log(2)+8[w]_{RH_1^{\prime\prime}}}.
$$
Thus, if $w\in RH''_{1}$ then $w\in RH_{1+\varepsilon}$ with the above choice of $\varepsilon$.

\end{proof}

\subsection{Proof of the Theorem \ref{Th:funnybound}} \label{Pr:funnybound}

We give a sketch of the proof in spirit of the proof of the 1-Gehring lemma.

Given a function $w\in RH_1$, we want to estimate $\ave{I}{w} \exp\left(-\ave{I}{(\log(w))}\right)$ from above and, therefore, we want to estimate $\ave{I}{(\log(w))}$ from below. Therefore, we denote
$$
B(x) = \inf \{ \ave{I}{(\log(w))}\colon \ave{I}{w}=x, \; \; \ave{I}{(w\log(w))} = y, \; \; [w]_{RH_1}\leqslant Q \}.
$$
The function $B$ is locally convex (we remind that both previous functions $B$ were locally concave, since we considered a $sup$ of something). We now denote by $\gamma_-$ the smaller root of the equation
$$
t-\log(t)=Q+1.
$$
We notice that for big $Q$ our $\gamma_-$ has the following asymptotic:
$$
\gamma_- \sim \frac{1}{e^{Q+1}}.
$$
We now define the function $v(x,y)$ by the equation
$$
y=(\log v + \gamma_-)x - v\gamma_-,
$$
$v\geqslant x$.

The picture is the following: we take a point $(x,y)$ and a tangent line to $\Gamma_Q = \{(x,y)\colon y=x\log(x) + Qx \}$, such that it ``kisses'' $\Gamma_Q$ on the left-hand side of $(x,y)$. Then $(v, v\log(v))$ is the point on the right-hand side of $(x,y)$, where this tangent hits $\Gamma = \{(x,y)\colon y=x\log(x)\}$.

Then the Bellman function $B$ is equal to
$$
B(x,y)=\log(v) + \frac{x-v}{\gamma_- v}.
$$

We skip all details of the proof since they are identical to both previous proofs.

We now notice that we are interested in the quantity
$$
x\exp(-B(x,y)).
$$
This is because $B(x,y)$ is the $\ave{I}{(\log(w))}$, and $x=\ave{I}{w}$. We write
$$
x\exp(-B(x,y)) = \frac{x}{v} \exp \left(\frac{1-\frac{x}{v}}{\gamma_-}\right) = s\exp\left(\frac{1-s}{\gamma_-}\right),
$$
where $s=\frac{x}{v} \in [\gamma_-, 1]$.
We set
$$
f(s)=s\exp\left(\frac{1-s}{\gamma_-}\right).
$$
Then
$$
f'(s) = \exp(\ldots) \left[ 1 - \frac{s}{\gamma_-}\right] \leqslant 0,
$$
so $f(s)\leqslant f(\gamma_-) = \gamma_- \exp\left(\frac{1-\gamma_-}{\gamma_-}\right)$.

Since $\gamma_- \sim e^{(-Q-1)}$, we get
$$
f(\gamma_-) \sim e^{(-Q-1)} e^{e^{Q+1}-1} = e^{e^{Q+1}-Q-2},
$$
which finishes our proof.

To illustrate that this result is sharp, we state the following proposition.
\begin{lemma}
Consider
$$
w(t)=\frac{1}{\gamma_-}t^{\frac{1-\gamma_-}{\gamma_-}}.
$$
Then $[w]_{RH_1}=Q$, and $\ave{I}{w}e^{-\ave{I}{(\log(w))}} = \gamma_- \exp\left(\frac{1-\gamma_-}{\gamma_-}\right)$.
\end{lemma}
\bibliographystyle{plain}
\fontsize{6}{5}\selectfont

\end{document}